\documentclass[11pt, a4paper]{amsart}

\usepackage{a4}
\usepackage{amssymb}
\usepackage{amsmath}
\usepackage{amsthm}
\usepackage{amstext}
\usepackage{amscd}
\usepackage{latexsym}
\usepackage{mathabx}
\usepackage{graphics}
\usepackage{color}
\usepackage[all]{xy}
\usepackage{verbatim}

\usepackage[colorlinks]{hyperref}
\hypersetup{
  colorlinks=true,
  citecolor=blue,
  linkcolor=blue,
  urlcolor=blue}

\newtheorem{thm}{Theorem}[section]
\newtheorem{prop}[thm]{Proposition}
\newtheorem{lem}[thm]{Lemma}
\newtheorem{cor}[thm]{Corollary}

\theoremstyle{definition}

\newtheorem{defn}[thm]{Definition}
\theoremstyle{remark}
\newtheorem{remk}[thm]{Remark}
\newtheorem{remks}[thm]{Remarks}

\newtheorem{exm}[thm]{Example}
\newtheorem{exms}[thm]{Examples}
\newtheorem{notat}[thm]{Notation}
\numberwithin{equation}{section}

{\hfill$\square$\end{defn}}
{\hfill$\square$\end{remk}}
{\hfill$\square$\end{remks}}
{\hfill$\square$\end{exm}}
{\hfill$\square$\end{exms}}
{\hfill$\square$\end{notat}}

\newcommand{\thmref}{Theorem~\ref}
\newcommand{\propref}{Proposition~\ref}
\newcommand{\corref}{Corollary~\ref}

\newcommand{\lemref}{Lemma~\ref}
\newcommand{\exref}{Example~\ref}

\newcommand{\sE}{{\mathcal E}}
\newcommand{\sF}{{\mathcal F}}
\newcommand{\sG}{{\mathcal G}}

\newcommand{\sI}{{\mathcal I}}

\newcommand{\sO}{{\mathcal O}}

\renewcommand{\P}{{\mathbb P}}

\newcommand{\codim}{{\rm codim}}

\newcommand{\Pic}{{\rm Pic}}

\newcommand{\Hom}{{\rm Hom}}

\newcommand{\sHom}{{\mathcal{H}{om}}}
\newcommand{\Lef}{{\rm L}}
\newcommand{\eLef}{{\rm eL}}

\newcommand{\ds}{{/\kern-3pt/}}

\renewcommand{\dim}{\text{\rm dim}}

\newcommand{\tuborg}{\left\{\begin{array}{ll}}
\newcommand{\sluttuborg}{\end{array}\right.}

\newcommand{\isoto}{\stackrel{\simeq}{\to}}

\renewcommand{\hat}{\widehat}

\renewcommand{\tilde}{\widetilde}

\begin{document}
\title[Equivariant Picard groups]
{A Grothendieck-Lefschetz theorem for Equivariant Picard groups}

\author{Charanya Ravi}
\address{Department of Mathematics, University of Oslo, P.O. Box 1053 Blindern, 
N-0316 Oslo, Norway}
\email{charanyr@math.uio.no}



\begin{abstract}
We prove a Grothendieck-Lefschetz theorem
for equivariant Picard groups of non-singular varieties with finite group actions.
\end{abstract}

\maketitle

\section{Introduction}
The geometry and $K$-theory of schemes with group scheme
actions have been extensively studied by various authors in recent years (e.g., see 
\cite{HVO15}, \cite{KR16}, \cite{H17}). 
The generalization of some of the fundamental theorems of algebraic geometry
to the equivariant setting has played an important role in the development of this subject.
The classical Lefschetz-type theorems compare the various algebraic invariants of 
non-singular projective varieties and their hyperplane sections. 
Let $X$ be a non-singular projective variety over a field $k$ of characteristic zero and
let $Y$ be a non-singular subvariety of $X$, of dimension $\geq 3$, which
is a scheme-theoretic complete intersection in $X$.
The Grothendieck-Lefschetz theorem for Picard groups 
(see \cite[Th\'eorem\`e XI.3.1]{SGA2}, \cite[Corollary IV.3.3]{Har70}) 
states that the Picard groups of $X$ and $Y$ are isomorphic.
The purpose of this article is to prove an analogous result for
varieties with finite group actions.

For a variety $X$ with $G$-action,
let $\Pic^G(X)$ denote the equivariant Picard group of $X$ (see \cite[1.3, page 32]{GIT}).
Our main result is the following.
\begin{thm} \label{cor:char0}
Let $k$ be a field of characteristic zero and let $G$ be a finite group.
Let $X$ be a non-singular projective variety over $k$ with $G$-action
and let $Y$ be a non-singular $G$-invariant subvariety of dimension $\geq 3$,
which is a scheme-theoretic complete intersection in $X$. Then the natural map
$
\Pic^G(X) \rightarrow \Pic^G(Y)
$
is an isomorphism.
\end{thm}

In view of the Kodaira-Akizuki-Nakano vanishing theorem, 
Theorem \ref{cor:char0} is a straightforward consequence of the technical result
Theorem \ref{thm:Main-thm}, 
which is proved by closely following the proof of the
Grothendieck-Lefschetz theorems
given in \cite[Chapter IV]{Har70}.
The main idea is to use the formal completion of $X$ along $Y$ and 
a suitable equivariant generalization of the Lefschetz conditions,
which is discussed in (\ref{sec:Lefconditions}).
As a corollary to Theorem \ref{thm:Main-thm}, 
we also deduce that
if $G$ acts on a projective space $X$ over $k$ (a field of arbitrary characteristic) 
and $Y$ is a $G$-invariant
scheme-theoretic complete intersection in $X$ such that $\dim(Y) \geq 3$, then the
equivariant Picard groups of $X$ and $Y$ are isomorphic (see Corollary \ref{cor:Pn}, 
\cite[Corollary IV.3.2]{Har70} for the non-equivariant case).

\subsection*{Acknowledgements}
It is a pleasure to thank Amalendu Krishna for suggesting the problem,
K. V. Shuddhodan for his lecture on the Grothendieck-Lefschetz theorems
and Anand Sawant, Paul Arne {\O}stv{\ae}r for their comments and suggestions 
which led to an improved exposition.
The writing of this paper was completed during the  
trimester program ``K-theory and Related Fields" at the
Hausdorff Research Institute for Mathematics and we 
thank the Institute and 
the organizers for their hospitality.
This work was supported by the RCN Frontier
Research Group Project no. 250399 ``Motivic Hopf equations".

\section{Preliminaries} \label{sec:Prelim}
We will work over a base field $k$ of arbitrary characteristic.
All schemes are assumed to be separated and of finite type over $k$.
The term variety will refer to an integral scheme over $k$.
In this note, $G$ will always denote a finite group.

\subsection{Group action on formal schemes} \label{sec:Group-action}
In this section, we recall briefly the notion of $G$-action on
a locally ringed space and equivariant sheaves. 
In the process we set up notations and terminologies
for the rest of the paper.

Let $(X, \sO_X)$ be a locally ringed space.
A {\sl $G$-action} on $(X, \sO_X)$ is a group homomorphism from $G$
to the group of automorphisms of $(X, \sO_X)$.
A morphism $\theta: (X, \sO_X) \to (X', \sO_{X'})$ of locally ringed spaces 
with $G$-actions is said to be {\sl $G$-equivariant} if it is compatible with the
$G$-actions on $(X, \sO_X)$ and $(X', \sO_{X'})$.

\begin{defn} \label{def:G-module}
Let $(X, \sO_{X})$ be a locally ringed space with a given $G$-action.
\begin{enumerate}
\item A {\sl $G$-sheaf} of abelian groups on $X$ is a
sheaf of abelian groups $\sF$ together with a collection of sheaf isomorphisms
$\phi_g : \sF \isoto g_*\sF$, for each $g \in G$,
which are subject to the conditions $\phi_e = id$ and 
$\phi_{gh} = h_*(\phi_g) \circ \phi_h$, for all $h \in G$.
We shall denote a $G$-sheaf in the sequel by $(\sF, \{\phi_g\})$.

\item A {\sl $G$-module} is a $G$-sheaf $(\sF, \{\phi_g\})$ 
such that $\sF$ is an $\sO_X$-module
and each $\phi_g$ is an $\sO_X$-module isomorphism. 
A {\sl locally free }(resp. {\sl invertible}) {\sl $G$-sheaf} is a
$G$-module $(\sF, \{\phi_g\})$, where $\sF$ is a locally free
(resp. invertible) sheaf of $\sO_X$-module.

\item A {\sl $G$-equivariant morphism} of $G$-sheaves
$f: (\sF, \{\phi_g\}) \to (\sG, \{\phi'_g\})$
is a morphism of sheaves $f: \sF \to \sG$ such that 
$\phi'_g \circ f = g_*(f) \circ \phi_g$, for all $g \in G$.
The set of $G$-equivariant morphisms from 
$\sF$ to $\sG$ is denoted by $\Hom^G(\sF, \sG)$.
If $\sF$ and $\sG$ are $G$-modules, the set of $G$-equivariant 
$\sO_X$-module homomorphisms is denoted by $\Hom^G_{\sO_X}(\sF, \sG)$.
\end{enumerate}
\end{defn}

\begin{exm} \label{ex:action-sch}
When $X$ is a $k$-scheme, a $G$-action on the locally ringed space $X$
defined as above coincides with the usual notion of group scheme action
on schemes, where $G$ is viewed as a finite constant group scheme over $k$. 
Let $\sigma: G \times X \to X$ denote the action map.
It is easy to verify that a $G$-module structure on a sheaf 
$F$ of $\sO_X$-modules as above
is equivalent to giving an isomorphism of $\sO_{G \times X}$-modules,
$\phi: \sigma^*F \to p_2^*F$, over $G \times X$.
Therefore $F$ is a $G$-module in the sense of \cite{GIT}.
\end{exm}

\begin{exm} \label{ex:Action-Formal-sch}
Let $X$ be a noetherian scheme with $G$-action and let $Y$ be a $G$-invariant 
closed subscheme, defined by a sheaf of ideals $\sI$ 
(which is a $G$-submodule of $\sO_X$).
Then $(\hat{X}, \sO_{\hat{X}})$, the formal completion of $X$ along $Y$,
has an induced $G$-action, as the direct image functor commutes 
with inverse limits. 
The canonical morphism $i: \hat{X} \to X$ is then $G$-equivariant.
Given a $G$-equivariant coherent $\sO_X$-module $F$,
the completion $\hat{F}$ of $F$ along $Y$, 
has a natural $G$-equivariant $\sO_{\hat{X}}$-module structure.
Furthermore, the functor $F \mapsto \hat{F}$ from the category of
coherent $\sO_X$-modules to the category of coherent $\sO_{\hat{X}}$-modules is exact
(see \cite[Corollary II.9.8]{Har13}) and therefore it is easy to see that it
induces an exact functor on the category of coherent $G$-modules.
\end{exm}

Let $(X, \sO_X)$ be a locally ringed space with $G$-action.
Let $Sh^G(X)$ denote the category of $G$-sheaves,
which is an abelian category with enough injectives.
Given a $G$-sheaf $\sF$ on $X$, 
the group $G$ acts on the global sections 
$\Gamma(X, \sF)$. Let $\Gamma(X, \sF)^G$ denote 
the $G$-invariant global sections, and
let $H^p(G; X, -)$ denote
the right derived functors of the 
functor $\Gamma(X, -)^G$. The groups $H^p(G; X, \sF)$ are the
$G$-cohomology groups of $\sF$.

\begin{lem} \label{lem:G-eq-hom=G-inv-hom}
Let $(X, \sO_X)$ be a locally ringed space with $G$-action
and let $(\sF, \{\phi_g\})$, $(\sG, \{\phi'_g\})$ be $G$-modules.
The sheaf $\sHom_{\sO_X}(\sF, \sG)$ 
has an induced $G$-module structure such that 
$
\Hom^G_{\sO_X}(\sF, \sG) 
= \Hom_{\sO_X}(\sF, \sG)^G.
$
\end{lem}

\begin{proof}
For each $g \in G$,
let $\rho_g: \sHom_{\sO_X}(\sF, \sG) \to \sHom_{\sO_X}(g_*\sF, g_*\sG)$
be the $\sO_X$-module homomorphism defined as follows.
Given an open subset $U$ of $X$ and $f \in Hom_{\sO_U}(\sF|_U, \sG|_U)$, let 
$\rho_g|_U(f) := (\phi'_g|_U) \circ f \circ (\phi_g^{-1}|_U).$
Let $\tilde{\rho_g} = \theta_g\circ \rho_g$,
where $\theta_g: \sHom_{\sO_X}(g_*\sF, g_*\sG)) {\overset{\theta_g} \to} 
g_*(\sHom_{\sO_X}(\sF, \sG))$ are the canonical isomorphisms.
Then $(\sHom_{\sO_X}(\sF, \sG), \{\rho_g\})$ is a $G$-module.
Now, 
\[
\begin{array}{llll}
f \in \Hom^G_{\sO_X}(\sF, \sG) & \Leftrightarrow & \rho_g(X)(f) = \phi'_g \circ f \circ 
(\phi_g)^{-1} = g_*(f), \forall g \in G &\\
&  \Leftrightarrow & \tilde{\rho_g}(X)(f) = f, \forall g \in G &\\
& \Leftrightarrow & f \in  \Hom_{\sO_X}(\sF, \sG)^G. &\\
\end{array}
\] 
\end{proof}

\begin{remk} \label{rem:G-sheaves-G-eq-hom}
If $\sF$ and $\sG$ are $G$-sheaves then one can
show using the same argument as above that $\sHom(\sF, \sG)$ is a 
$G$-sheaf and $\Hom^G(\sF, \sG) = \Hom(\sF, \sG)^G$.
\end{remk}

\begin{cor} \label{cor:nw-van-eq-sec}
Let $(X, \sO_X)$ be a locally ringed space with $G$-action and let
$\sF$ be an invertible $G$-sheaf on $X$. There is a $G$-equivariant isomorphism
$\sO_X \xrightarrow{\simeq} \sF$,
where $\sO_X$ has the canonical $G$-action, 
if and only if $\Gamma(X, \sF)^G$ has a nowhere vanishing section.
\end{cor}

\begin{proof}
The proof follows from \lemref{lem:G-eq-hom=G-inv-hom}, since
$\Hom_{\sO_X}(\sO_X, \sF) = \Gamma(X, \sF)$ as $G$-sets
and isomorphisms $\sO_X \to \sF$ correspond to nowhere vanishing
sections in $\Gamma(X, \sF)$.
\end{proof}

\subsection{The equivariant Lefschetz Conditions} \label{sec:Lefconditions}
In \cite[Section X.2]{SGA2}, 
Grothendieck introduced the Lefschetz conditions for pairs $(X,Y)$, inspired by Lefschetz,
where $X$ is a locally noetherian scheme and $Y$ is a closed subscheme of $X$.
These were essential in the proof of Grothendieck's theorems comparing
the Picard groups and the fundamental groups of a projective 
variety $X$ with a complete intersection subvariety $Y$.
For schemes with action of a finite group $G$, we define the equivariant
Lefschetz conditions in this section and prove equivariant analogues of
some results in the Grothendieck-Lefschetz theory
that will be useful in the sequel.

\begin{defn} \label{def:Lef-condition}
Let $X$ be a noetherian scheme with $G$-action, and let
$Y \subseteq X$ be a $G$-invariant closed subscheme.
Let $\hat{X}$ be the formal completion of $X$ along $Y$.
Then $\hat{X}$ is a locally ringed space with $G$-action as discussed in
\exref{ex:Action-Formal-sch}. 

\begin{enumerate}
\item The pair $(X, Y)$ satisfies the
{\sl equivariant Lefschetz condition}, 
written $\Lef^G(X,Y)$, if for every $G$-invariant open set $U \supseteq Y$,
and every $G$-equivariant locally free sheaf $E$ on $U$,
there exists a $G$-invariant open set $U'$ with $Y \subseteq U' \subseteq U$
such that the natural map
$
\Gamma(U', E|_{U'})^G \isoto \Gamma(\hat{X}, \hat{E})^G
$
is an isomorphism.

\item The pair $(X,Y)$ satisfies the 
{\sl equivariant effective Lefschetz condition},
written $\eLef^G(X,Y)$, if $\Lef^G(X,Y)$ holds, and in addition,
for every $G$-equivariant locally free sheaf $\sE$ on $\hat{X}$,
there exists a $G$-invariant open set $U \supseteq Y$ and a 
$G$-equivariant locally free sheaf $E$ on $U$ such that 
$\hat{E} \simeq \sE$
as $G$-modules.
\end{enumerate}
\end{defn}

With $(X,Y)$ as above,
let $E$ and $F$ be locally free $G$-sheaves defined on
$G$-invariant open neighbourhoods $U$ and $V$ of $Y$, respectively.
We write $E \sim F$ if there exists a $G$-invariant
open set $W$ with $Y \subseteq W \subseteq U \cap V$ such that
$E|_W \simeq F|_W$ as $G$-sheaves.
We define the category ${\rm LF}^0_G(Y)$ of germs of 
locally free $G$-sheaves around $Y$ as follows.
An object of this category is a class of locally free $G$-sheaves
defined on $G$-invariant open neighbourhoods of $Y$ under the equivalence 
relation $\sim$. For any two objects $[E]$ and $[F]$ in ${\rm LF}^0_G(Y)$,
the set of homomorphisms from $[E]$ to $[F]$ is defined to be the set
$\varinjlim_{U} ~ \Hom_{\sO_U}^G(E,F)$, where the colimit is
taken over all $G$-invariant open neighbourhoods $U$ of $Y$ such that both
$E$ and $F$ are defined over $U$.
Let ${\rm LF}_G(\hat{X})$ denote the category of 
locally free $G$-sheaves on $\hat{X}$.

\begin{lem} \label{lem:Germs-G-V.Bs}
Let $\wedge : {\rm LF}^0_G(Y) \rightarrow {\rm LF}_G(\hat{X})$
be the functor sending 
$E \mapsto \hat{E}$.
\begin{enumerate}
\item If $\Lef^G(X,Y)$ holds, then $\wedge$ is fully faithful.
\item If $\eLef^G(X,Y)$ holds, then $\wedge$ is an equivalence of 
categories.
\end{enumerate}
\end{lem}

\begin{proof}
Suppose $\Lef^G(X,Y)$ holds. Let $E, F \in {\rm LF}^0_G(Y)$. 
Without loss of generality, we may assume that $E, F$ are 
$G$-equivariant locally free $\sO_V$-modules for some $G$-invariant open
neighbourhood $V$ of $Y$.
Let $U$ be any $G$-invariant open neighbourhood of $Y$ 
such that $U \subseteq  V$.
$\sHom_{\sO_U}(E, F)$ is a $G$-equivariant locally free $\sO_U$-module,
by \lemref{lem:G-eq-hom=G-inv-hom}.
Since $\Lef^G(X,Y)$ holds, there exists a $G$-invariant open set $U'$ with
$Y \subseteq U' \subseteq U$ such that the natural map
$
\Gamma(U', \sHom_{\sO_U}(E, F))^G \isoto 
\Gamma(\hat{X}, \sHom_{\sO_U}(E, F)^{\wedge})^G
$
is an isomorphism.
Since $\sHom_{\sO_U}(E, F)^{\wedge} \isoto 
\sHom_{\sO_{\hat{X}}}(\hat{E}, \hat{F})$ as $G$-sheaves, 
$\Gamma(\hat{X}, \sHom_{\sO_U}(E, F)^{\wedge})^G
\isoto \Gamma(\hat{X}, \sHom_{\sO_{\hat{X}}}(\hat{E}, \hat{F}))^G$ is an isomorphism and
hence $\Gamma(U', \sHom_{\sO_U}(E, F))^G \isoto 
\Gamma(\hat{X}, \sHom_{\sO_{\hat{X}}}(\hat{E}, \hat{F}))^G$ is an isomorphism.
By \lemref{lem:G-eq-hom=G-inv-hom}, the above isomorphism can be
rewritten as 
$\Hom^G_{\sO_{U'}}(E,F) \isoto \Hom^G_{\sO_{\hat{X}}}(\hat{E}, \hat{F}).$
This proves that the functor $\wedge$ is fully faithful.
 If $\eLef^G(X,Y)$ holds, $\wedge$ is further essentially
surjective (by definition), and therefore an equivalence of categories.
\end{proof}

\begin{prop} \label{prop:Leff-cond-complete-int}
Let $X$ be a non-singular projective variety with $G$-action.
Let $Y \subseteq X$ be a $G$-invariant closed subscheme, which is 
a scheme-theoretic complete intersection in $X$.
If $\dim(Y) \geq 2$,
then $\eLef^G(X,Y)$ holds.
\end{prop}

\begin{proof}
Let $U \supseteq Y$ be any $G$-invariant open set, and let $E$
be a locally free $G$-sheaf on $U$.
Since $Y$ is a complete intersection, by
\cite[Corollary IV.1.2]{Har70} and the proof of \cite[Proposition IV.1.1]{Har70},
the $G$-equivariant restriction map
$
\Gamma(U, E) \isoto \Gamma(\hat{X}, \hat{E})
$
is an isomorphism.
This induces an isomorphism
$
\Gamma(U, E)^G \isoto \Gamma(\hat{X}, \hat{E})^G.
$
Therefore $\Lef^G(X,Y)$ holds.

Let $\hat{X}$ be the formal completion of $X$ along $Y$,
and let $(\sE, \{\phi_g\})$ be a locally free $G$-sheaf on $\hat{X}$.
Since $Y$ is a scheme-theoretic local complete intersection, by \cite[Theorem IV.1.5]{Har70},
we can find an open set $U \supseteq Y$ (not necessarily $G$-invariant) 
and a locally free sheaf $E$ on $U$ such that $\theta: \hat{E} \isoto \sE$ 
non-equivariantly. We may assume that $U$ is $G$-invariant by replacing
$U$ by the open set $\bigcap_{g \in G} gU$.
For each $g \in G$, $g_*E$ is then a locally free sheaf on $U$ such that
we have induced isomorphisms $\hat{g_* E} \simeq g_*\sE$, 
since direct image functor commutes with inverse limits.
Since $E, g_*E$ are locally free sheaves on $U$,  $\sHom_{\sO_U}(E, g_*E)$
is a locally free $\sO_U$-module.
Again as above, we have isomorphisms
$
\Hom_{\sO_U}(E, g_*E) \isoto \Hom_{\hat{X}}(\hat{E}, g_*{\hat{E}}) 
\isoto \Hom_{\hat{X}}(\sE, g_* \sE)
$
for each $g \in G$. Therefore, 
$\phi_g \in \Hom_{\hat{X}}(\sE, g_* \sE)$ can be uniquely lifted to
a morphism $\tilde{\phi}_g \in \Hom_{\sO_U}(E, g_*E)$. 
Since the lifts are unique and $\{\phi_g\}_{g \in G}$ defines a
$G$-module structure on $\sE$,
$\{\tilde{\phi}_g\}_{g \in G}$ defines a $G$-module structure on $E$.
Further $\theta : \hat{E} \to \sE$ is a $G$-equivariant morphism, by definition
of the $G$-action on $E$. Therefore, $\eLef^G(X, Y)$ holds.
\end{proof}

\section{Equivariant Grothendieck-Lefschetz theorem} \label{sec:Main-thm}
We prove Theorem \ref{cor:char0} in this section. 
The following Lemma
identifying the equivariant Picard groups of a variety $X$
and its formal completion $\hat{X}$
will be crucial for proving our main result.

\begin{lem} \label{lem:PicG(X)=PicG(X^)}
Let $X$ be a non-singular variety with $G$-action
and let $Y \subseteq X$ be a $G$-invariant closed subscheme
such that $Y$ meets every effective divisor on $X$.
Let $\hat{X}$ denote the completion of $X$ along $Y$ with the induced
$G$-action. 
Assume that $\dim(X) \geq 2$ and $\eLef^G(X,Y)$ holds.
Then the canonical map
$
\Pic^G(X) \to \Pic^G(\hat{X})
$
is an isomorphism.
\end{lem}

\begin{proof}
Since $\eLef^G(X,Y)$ holds, every invertible $G$-sheaf on $\hat{X}$ extends uniquely 
to an invertible $G$-sheaf on some $G$-invariant open neighbourhood $U$
of $Y$ by \lemref{lem:Germs-G-V.Bs}. 
Since $Y$ meets every effective divisor on $X$,
we have $\codim(X-U, X) \geq 2$. 
Therefore by \cite[Lemma~2(1)]{EG98}, $\Pic^G(X) \to \Pic^G(U)$
is an isomorphism. The canonical morphism $\Pic^G(X) \to \Pic^G(\hat{X})$
factors through $\Pic^G(U)$ for every $G$-invariant open 
$U$ such that $Y \subseteq U$. Hence we conclude that 
$\Pic^G(X) \to \Pic^G(\hat{X})$ is an isomorphism.
\end{proof}

\begin{lem} \label{lem:PicG(X^)=inv. lim. PicG(Yn)}
Let $X$ be a proper scheme with $G$-action
and let $Y \subseteq X$ be a $G$-invariant closed subscheme
defined by a $G$-sheaf of ideals $\sI$.
For $n \geq 1$, let $Y_n$ denote the $G$-invariant closed subscheme
defined by the sheaf of ideals $\sI^n$. Then
$
\Pic^G(\hat{X}) \simeq \varprojlim_{n} ~ \Pic^G(Y_n).
$
\end{lem}

\begin{proof}
If $\sF$ is an invertible $G$-sheaf on $\hat{X}$, then 
$F_n = \sF \otimes_{\sO_{\hat{X}}}  \sO_{Y_n}$ is an invertible
$G$-sheaf on $Y_n$. This defines a map
$
f: \Pic^G(\hat{X}) \to \varprojlim_{n} ~ \Pic^G(Y_n).
$

An element of $  \varprojlim_{n} ~ \Pic^G(Y_n)$ is given by a 
collection of invertible $G$-sheaves $F_n$ on $\Pic^G(Y_n)$ along with
$G$-equivariant isomorphisms 
$F_{n+1} \otimes_{\sO_{Y_{n+1}}} \sO_{Y_n} \isoto F_n.$
Composing with the natural $G$-equivariant map 
$F_{n+1} \to F_{n+1}  \otimes_{\sO_{Y_{n+1}}} \sO_{Y_n}$,
we get a projective system of invertible $G$-sheaves of 
$\sO_{\hat{X}}$-modules. Then $\sF = \varprojlim_{n} F_n$
is an invertible $G$-sheaf on $\hat{X}$ with
$ \sF \otimes_{\sO_{\hat{X}}} \sO_{Y_n} \simeq F_n$.
Therefore $f$ is surjective. 
To see that $f$ is injective, let $\sF$ be an invertible $G$-sheaf on $\hat{X}$
such that for each $n$, there is a $G$-equivariant isomorphism 
$\sF \otimes_{\sO_{\hat{X}}}  \sO_{Y_n} \isoto \sO_{Y_n}$,
where $\sO_{Y_n}$ has the canonical $G$-action.
By \cite[Proposition II.9.2]{Har13}
and since $(-)^G$ is an additive left exact functor preserving products,
it follows that the functor $\Gamma(Y,-)^G$ preserves inverse limits.
Therefore
$\Gamma(\hat{X}, \sF)^G = \varprojlim_{n} \Gamma(Y_n, F_n)^G$,
where $F_n := \sF \otimes_{\sO_{\hat{X}}}  \sO_{Y_n}$
and the inverse system $\Gamma(Y_n, F_n)^G$ satisfies the Mittag-Leffler
condition \cite[Chapter 0, 13.1.2]{EGA3.1} 
(since $Y_n$ is proper, $\Gamma(Y_n, F_n)^G$ is a finite-dimensional
$k$-vector space). 
By \corref{cor:nw-van-eq-sec}, each $F_n$ has a nowhere vanishing $G$-invariant
section. Therefore the stable images in the inverse system have nowhere vanishing
sections, so we can find a nowhere vanishing section $s \in \Gamma(\hat{X}, \sF)^G$.
Therefore, again by \corref{cor:nw-van-eq-sec}, $\sF \simeq \sO_{\hat{X}}$
is trivial.
\end{proof} 

\begin{thm} \label{thm:Main-thm}
Let $k$ be a field and let $G$ be a finite group.
Let $X$ be a proper non-singular variety over $k$ with $G$-action
and let $Y \subseteq X$ be a $G$-invariant closed subscheme
defined by a $G$-sheaf of ideals $\sI$. Suppose that
\begin{enumerate}
\item $\eLef^G(X,Y)$ holds (see Definition \ref{def:Lef-condition}(2));
\item $Y$ meets every effective divisor on $X$; and
\item $H^i(G; Y, I^n/I^{n+1}) = 0$ for $i = 1, 2$ for all $n \geq 1$.
\end{enumerate}
Then the natural map
$
\Pic^G(X) \rightarrow \Pic^G(Y)
$
is an isomorphism.
\end{thm}

\begin{proof} 
The natural map in question factors as
$
\Pic^G(X) \xrightarrow{\alpha} \Pic^G(\hat{X}) \xrightarrow{\beta} \Pic^G(Y),
$
where $\alpha$ and $\beta$ are the natural restriction maps.
The map $\alpha$ is an isomorphism by \lemref{lem:PicG(X)=PicG(X^)}.
Factorise the map $\beta$ further as follows.
For $n \geq 1$, let $Y_n$ denote the $G$-invariant closed subscheme
defined by the sheaf of ideals $\sI^n$. We have natural maps:
$$
\Pic^G(\hat{X}) \to \varprojlim\nolimits_n~\Pic^G(Y_n) \to \cdots \to \Pic^G(Y_{n+1})
 \to \Pic^G(Y_{n}) \to \cdots \to \Pic^G(Y).
$$
We will show that all the above maps are isomorphisms.
The first map is an isomorphism by \lemref{lem:PicG(X^)=inv. lim. PicG(Yn)}.
Let $n \geq 1$ and consider the exact sequence of
$G$-sheaves
$
0 \to I^n/I^{n+1} \to \sO_{Y_{n+1}}^* \to \sO_{Y_{n}}^* \to 0,
$
where $\sO^*$ denotes the multiplicative group of units 
and the first map is given by $x \mapsto 1+x$.
This gives a long exact sequence of $G$-cohomology groups:
\begin{align*}
\cdots \to H^1(G; Y, I^n/I^{n+1}) \to H^1(G; Y_{n+1}, \sO_{Y_{n+1}}^*) 
& \to H^1(G; Y_{n}, \sO_{Y_{n}}^*) \\
& \to H^2 (G; Y, I^n/I^{n+1}) \to \cdots.
\end{align*}
By hypothesis $(3)$, we conclude that 
$H^1(G; Y_{n+1}, \sO_{Y_{n+1}}^*) \isoto H^1(G; Y_{n}, \sO_{Y_{n}}^*)$.
By \cite[Theorem 2.7]{HVO15}, this shows that
$\Pic^G(Y_{n+1}) \to \Pic^G(Y_{n})$ is an isomorphism.
Consequently,
$\varprojlim_{n}~\Pic^G(Y_n)$ is isomorphic to $\Pic^G(Y_n)$ for
every $n \geq 1$.
This completes the proof of the theorem.
\end{proof}

\begin{cor} \label{cor:Pn}
Suppose $G$ acts on $\P_k^n$ and
$Y$ is a $G$-invariant closed subscheme of dimension $\geq 3$
which is a scheme-theoretic complete intersection in $\P_k^n$. Then
the natural map
$
\Pic^G(\P^n_k) \rightarrow \Pic^G(Y)
$
is an isomorphism.
\end{cor}

\begin{proof}
Since $Y$ is a $G$-invariant scheme-theoretic complete intersection and $\dim(Y) \geq 3$,
$\eLef^G(X,Y)$ holds by \propref{prop:Leff-cond-complete-int} and
$Y$ meets every effective divisor on $\P^n_k$ 
by \cite[Theorem III.5.1, Proposition IV.1.1]{Har70}. 
Further if $Y$ is an intersection
of hypersurfaces of degree $d_1, \cdots, d_r$ then 
$I/I^2 \simeq \sO_Y(-d_1) \oplus \cdots \oplus \sO_Y(-d_r)$. Hence
for all $n \geq 1$,
$I^n/I^{n+1}$ is a direct sum of sheaves of the form $\sO_Y(m)$
for suitable integers $m < 0$. By \cite[Proposition 5]{FAC}, $H^i(Y, \sO_Y(m)) = 0$
for all $0 \leq i < \dim(Y)$ for $m < 0$. Since $\dim(Y) \geq 3$,
$H^i(Y, I^n/I^{n+1}) = 0$ for $0 \leq i \leq 2$. Therefore by 
\cite[(2.5)]{HVO15}, $H^i(G; Y, I^n/I^{n+1}) = 0$ for $i = 1, 2$.
This shows that the hypotheses of \thmref{thm:Main-thm} are satisfied. 
\end{proof}

\begin{proof} [Proof of Theorem \ref{cor:char0}]
It is enough to check as in the above corollary that $H^i(Y, \sO_Y(m)) = 0$
for $0 \leq i \leq 2$ and all $m < 0$. This follows from 
the Kodaira-Akizuki-Nakano vanishing theorem (see \cite[Corollary 2.11]{DI87}) 
as $\dim(Y) \geq 3$.
\end{proof}

\section*{}


\begin{thebibliography}{99}
\bibitem{DI87} P. Deligne, L. Illusie, 
{\sl Rel\`evements modulo $p^2$ et d\'ecomposition du complexe de de Rham\/}, 
Invent. Math. {\bf 89}, (1987), no.~2, 247--270.\

\bibitem{EG98} D. Edidin, W. Graham, 
{\sl Equivariant intersection theory\/}, 
Invent. Math. {\bf 131}, (1998), no.~3, 595--634.\

\bibitem{EGA3.1} A. Grothendieck, 
{\sl \'El\'ements de g\'eom\'etrie alg\'ebrique. III. \'Etude cohomologique des faisceaux 
coh\'erents. I \/}, 
Inst. Hautes \'Etudes Sci. Publ. Math. No. 11, (1961), 167 pp. \

\bibitem{SGA2} A. Grothendieck, 
{\sl Cohomologie locale des faisceaux coh\'erents et th\'eor\`emes de Lefschetz locaux et 
globaux $(SGA$ $2)$\/}, 
North-Holland Publishing Co., Amsterdam, (1968).\

\bibitem{Har13} R. Hartshorne, 
{\sl Algebraic Geometry\/}, 
Graduate text in Math., {\bf 52}, Springer, (1997). \

\bibitem{Har70} R. Hartshorne, 
{\sl Ample subvarieties of algebraic varieties\/}, 
Lecture Notes in Mathematics, Vol. 156, Springer-Verlag, Berlin, (1970).\

\bibitem{HVO15} J. Heller, M. Voineagu, P. A. {\O}stv{\ae}r, 
{\sl Equivariant cycles and cancellation for motivic cohomology\/}, 
Doc. Math. {\bf 20}, (2015), 269--332.\

\bibitem{H17} M. Hoyois, 
{\sl The six operations in equivariant motivic homotopy theory\/}, 
Adv. Math. {\bf 305}, (2017), 197--279. \

\bibitem{KR16} A. Krishna, C. Ravi, 
{\sl Equivariant vector bundles, their derived category and $K$-theory on affine schemes\/}, 
Ann. K-Theory {\bf 2}, (2017), no.~2, 235--275.\

\bibitem{GIT} D. Mumford, J. Fogarty, F. Kirwan, 
{\sl Geometric Invariant Theory \/}, 
Third Edition, Results of Mathematics and its Boundaries (2), 34, Springer-Verlag, Berlin, (1994).\

\bibitem{FAC} J.-P. Serre, 
{\sl Faisceaux alg\'ebriques coh\'erents\/}, Ann. of Math. (2) {\bf 61}, (1955), 197--278.\ 



\end{thebibliography}
\end{document}